\documentclass{amsart}
\usepackage{amssymb, amsmath, amsthm, graphics, comment, xspace, enumerate}
\newtheorem{theorem}{Theorem}[section]

\newtheorem{lemma}[theorem]{Lemma}
\newtheorem{proposition}[theorem]{Proposition}
\theoremstyle{definition}
\newtheorem{definition}[theorem]{Definition}

\theoremstyle{remark}
\newtheorem{remark}[theorem]{Remark}
\numberwithin{equation}{section}

\begin{document}
\title{Subspace Almost Periodic $C$-Distribution Semigroups and $C$-Distribution Cosine Functions}
\author{Marko Kosti\' c}
\address{Faculty of Technical Sciences,
University of Novi Sad,
Trg D. Obradovi\' ca 6, 21125 Novi Sad, Serbia}
\email{marco.s@verat.net}

\author{Stevan Pilipovi\' c}
\address{Department for Mathematics and Informatics,
University of Novi Sad,
Trg D. Obradovi\' ca 4, 21000 Novi Sad, Serbia}
\email{pilipovic@dmi.uns.ac.rs}

\author{Daniel Velinov}
\address{Department for Mathematics, Faculty of Civil Engineering, Ss. Cyril and Methodius University, Skopje,
Partizanski Odredi
24, P.O. box 560, 1000 Skopje, Macedonia}
\email{velinovd@gf.ukim.edu.mk}

\newcommand{\FilMSC}{Primary 47D06, 47D60; Secondary 47B37, 47D60, 47D62.}
\newcommand{\FilKeywords}{$C$-distribution semigroups, $C$-distribution cosine functions, integrated $C$-semigroups, integrated $C$-cosine functions, subspace almost periodicity}
\newcommand{\FilCommunicated}{Dragan S. Djordjevi\'c}
\newcommand{\FilSupport}{Research supported by
grant 174024 of Ministry of Science and Technological Development, Republic of Serbia.}

\begin{abstract}
The main aim of this paper is to introduce and analyze the notions of subspace almost periodicity and subspace weak almost periodicity for $C$-distribution semigroups and
$C$-distribution cosine functions in Banach spaces. We continue our previous research study of almost periodicity of abstract Volterra integro-differential equations \cite{aot}, focusing our attention on the abstract ill-posed Cauchy problems of first order.
\end{abstract}

\maketitle

\makeatletter
\renewcommand\@makefnmark%
{\mbox{\textsuperscript{\normalfont\@thefnmark)}}}
\makeatother

\section{Introduction and preliminaries}\label{intro}
As mentioned in the abstract, we continue our previous research study \cite{aot} by enquiring into the basic subspace almost periodic and subspace weak almost periodic properties of $C$-distribution semigroups and
$C$-distribution cosine functions in Banach spaces. The introduction of notions of subspace almost periodicity and subspace weak almost periodicity is motivated by the fact that the integral generators of almost periodic strongly continuous semigroups and almost periodic  integrated $C$-semigroups explored in the existing literature need to satisfy rather restrictive spectral conditions (since we primarily deal with semigroups and cosine operator functions consisted of unbounded linear operators, the notion of subspace uniform almost periodicity will not attract our attention here). On the other hand, a great number of chaotic or subspace chaotic $C$-distribution semigroups and $C$-distribution cosine functions satisfying certain kinds of the Desch-Schappacher-Webb criterion are subspace almost periodic, as well, with the subspace of almost periodicity being generally dense in the initial Banach space. Because of that, we can freely say that the notion of subspace almost periodicity is incredibly important considered from the application's point of view.

The classes of (bounded) almost periodic distribution groups and cosine distributions were considered for the first time by I. Cioranescu \cite{ioana} in 1990. Six years later, in 1996, Q. Zheng and L. Liu \cite{zhengliu} investigated the class of almost periodic tempered distribution semigroups. Almost periodic distribution (semi-)groups and cosine distributions considered in \cite{ioana} and \cite{zhengliu} are exponential and have densely defined integral generators, which is not the case with $C$-distribution semigroups and
$C$-distribution cosine functions considered in this paper.
The reader with a little experience will easily observe that our approach is completely different from those employed in \cite{ioana} and \cite{zhengliu}: speaking only in terms of global $n$-times integrated semigroups and cosine functions, here we are interested in question whether $n$-th derivatives of such semigroups and cosine functions exist and are almost periodic on a certain subspace $\tilde{E}$ of the pivot space $E$ (the main objective in \cite{ioana} and \cite{zhengliu} is to analyze the almost periodicity of induced $n$-times integrated semigroups and cosine functions, see e.g. \cite[Theorem 1.1(ii)]{ioana}, which is very difficult to be satisfied in the case that $n\in {\mathbb N}$). The paper is consisted from two sections; in the first one, we collect some preliminaries concerning $C$-distribution semigroups and $C$-distribution cosine functions, integrated
$C$-semigroups and integrated $C$-cosine functions, as well as vector-valued almost periodic functions, while in the second one we formulate and prove our main structural results, proposing also some open problems we have not been able to solve.

We use the standard notation throughout the paper.
Unless specifed otherwise,
we shall always assume henceforth
that $(E,\| \cdot \|)$ is a complex Banach space. If $X$ and $Y$ are general vector topological spaces, then
$L(X,Y)$ designates the space of all continuous linear mappings from $X$ into
$Y;$ $L(E)\equiv L(E,E).$ If $A$ is a closed linear operator
acting on $E,$
then the domain, kernel space and range of $A$ will be denoted by
$D(A),$ $N(A)$ and $R(A),$
respectively. Since no confusion
seems likely, we will identify $A$ with its graph. The Banach space
$D(A)$ equipped with the graph norm $\|x\|_{[D(A)]}:=\|x\|+\|Ax\|,$ $x\in D(A)$ will be denoted by $[D(A)]$.
By $E^{\ast}$ we denote the dual space of $E;$ $C\in L(E)$ will be injective and the inclusion
$CA\subseteq AC$ will be assumed.
Set $g_{\zeta}(t):=t^{\zeta-1}/\Gamma(\zeta),$ $t>0$ ($\zeta>0$), where $\Gamma(\cdot)$ denotes the Gamma function.

Now we will recall the basic facts about vector-valued distribution spaces used henceforth.
The Schwartz spaces \index{space!Schwartz} of test functions $\mathcal{D}=C_0^{\infty}(\mathbb{R})$
and $\mathcal{E}=C^{\infty}(\mathbb{R})$
are equipped with the usual inductive limit topologies; the topology
of space of rapidly decreasing functions $\mathcal{S}$ defines the following system of seminorms
$
p_{m,n}(\psi):=\sup_{x\in\mathbb{R}}|x^m\psi^{(n)}(x)|,$ $
\psi\in\mathcal{S},\ m,\ n\in\mathbb{N}_0.
$
If $\emptyset \neq \Omega  \subseteq {\mathbb R},$ then by $\mathcal{D}_{\Omega}$ we denote the subspace of $\mathcal{D}$ consisting of those functions $\varphi \in \mathcal{D}$ for which supp$(\varphi) \subseteq \Omega;$ $\mathcal{D}_{0}\equiv \mathcal{D}_{[0,\infty)}.$
If $\varphi$, $\psi:\mathbb{R}\to\mathbb{C}$ are
measurable functions, the convolution products $\varphi*\psi$
and $\varphi*_0\psi$ are defined by
$$
\varphi*\psi(t):=\int\limits_{-\infty}^{\infty}\varphi(t-s)\psi(s)\,ds\mbox{ and }
\varphi*_0
\psi(t):=\int\limits^t_0\varphi(t-s)\psi(s)\,ds,\;t\;\in\mathbb{R}.
$$
If $\varphi\in\mathcal{D}$ and $f\in\mathcal{D}'$, or $\varphi\in\mathcal{E}$ and $f\in\mathcal{E}'$,
then we define the convolution $f*\varphi$ by $(f*\varphi)(t):=f(\varphi(t-\cdot))$, $t\in\mathbb{R}$.
For $f\in\mathcal{D}'$, or for $f\in\mathcal{E}'$,
define $\check{f}$ by $\check{f}(\varphi):=f(\varphi (-\cdot))$, $\varphi\in\mathcal{D}$ ($\varphi\in\mathcal{E}$).
In general, the convolution of two distributions $f$, $g\in\mathcal{D}'$, denoted by $f*g$,
is defined by $(f*g)(\varphi):=g(\check{f}*\varphi)$, $\varphi\in\mathcal{D}$.
It is well-known that $f*g\in\mathcal{D}'$ and supp$(f*g)\subseteq$supp$ (f)+$supp$(g)$. For every $t\in {\mathbb R},$ we define the Dirac distribution centered at point $t,$
$\delta_{t}$ for short, by $\delta_{t}(\varphi):=\varphi(t),$ $\varphi \in {\mathcal D}.$

The space
$\mathcal{D}'(E):=L(\mathcal{D},E)$ is consisted of all
continuous linear functions ${\mathcal D}\rightarrow E;$
$\mathcal{D}'_{\Omega}(E)$ denotes the subspace of
$\mathcal{D}'(E)$ containing $E$-valued distributions \index{vector-valued!distributions}
whose supports are contained in $\Omega .$ Set $\mathcal{D}'_{0}(E):= \mathcal{D}'_{[0,\infty)}(E).$
If $E={\mathbb C},$ then the above spaces are also denoted by $\mathcal{D}',$
$\mathcal{D}'_{\Omega}$ and
$\mathcal{D}_0'.$ For more details about vector-valued distributions, we refer the reader to L. Schwartz \cite{sch166}-\cite{sch167}; the convolution of
vector-valued distributions will be taken in the sense of
\cite[Proposition 1.1]{ku112}.

Now we recall the definition of a $C$-distribution semigroup (see \cite{knjigah}):

\begin{definition}\label{c-ds-frim}
Let $\mathcal{G}\in\mathcal{D}_0^{\prime}(L(E))$ satisfy $C\mathcal{G}=\mathcal{G}C$.
Then it is said that  $\mathcal{G}$ is a $C$-distribution semigroup, shortly \emph{(C-DS)}, iff ${\mathcal G}$ satisfies the following conditions:
\begin{itemize}
\item[(i)] ${\mathcal G}(\varphi\ast_0\psi)C={\mathcal G}(\varphi){\mathcal G}(\psi)$, for any $\varphi,\ \psi\in\ {\mathcal D}$.
\item[(ii)] ${\mathcal N}({\mathcal G}):=\bigcap_{\varphi\in\ {\mathcal D}_0}N({\mathcal G}(\varphi))=\{0\}$.\\
A \emph{(C-DS)} ${\mathcal G}$ is called dense iff, in addition to the above,
\item[(iii)] ${\mathcal R}(\mathcal{G}):=\bigcup_{\varphi\in\mathcal{D}_0}R(\mathcal{G}(\varphi))$
is dense in $E.$
\end{itemize}
\end{definition}

Let $\mathcal{G}\in\mathcal{D}_0'(L(E))$ be a (C-DS) and let $T\in\mathcal{E}_0',$
i.e., $T$ is a scalar-valued distribution with compact support contained in $[0,\infty)$.
Define
\[
G(T):=\Bigl\{(x,y) \in E\times E : \mathcal{G}(T*\varphi)x=\mathcal{G}(\varphi)y\;\mbox{ for all }\;\varphi\in\mathcal{D}_0 \Bigr\}.
\]
Then  it can be easily seen that $G(T)$ is a closed linear operator commuting with $C$.
We define the (infinitesimal) generator $A$ of a pre-(C-DS) $\mathcal{G}$ by $A:=G(-\delta').$ We know that $C^{-1}AC=A$ as well as that the following holds:
Let $S$, $T\in\mathcal{E}'_0$, $\varphi\in\mathcal{D}_0$,  $\psi\in\mathcal{D}$ and $x\in E$.
Then we have:
\begin{itemize}
\item[A1.] $G(S)G(T)\subseteq G(S*T)$ with $D(G(S)G(T))=D(G(S*T))\cap D(G(T))$, and $G(S)+G(T)\subseteq G(S+T)$.
\item[A2.] $(\mathcal{G}(\psi)x$, $\mathcal{G}(-\psi^{\prime})x-\psi(0)Cx)\in A$.
\end{itemize}

We denote by $D({\mathcal G})$ the set consisting of those elements $x\in E$ for which $x\in D(G({\delta}_t)),$ $t\geq 0$ and the mapping $t\mapsto G({\delta}_t)x,$ $t\geq 0$ is continuous. By A1., we have
that
\begin{align*}
D\bigl(G(\delta_s)G(\delta_t)\bigr)\!=\!D\bigl(G(\delta_s*\delta_t)\bigr)\cap D\bigl(G(\delta_t)\bigr)\!=\!D\bigl(G(\delta_{t+s})\bigr)\cap D\bigl(G(\delta_t)\bigr),
\;t,\,s\!\geq\! 0,
\end{align*}
which clearly implies $G(\delta_t)(D(\mathcal{G}))\subseteq D(\mathcal{G})$, $t\geq 0$.

We refer the reader to \cite{a43}, \cite{l1} and \cite{knjigah}-\cite{knjigaho} for further information concerning fractionally integrated $C$-semigroups and fractionally integrated $C$-cosine functions in Banach spaces, properties of their subgenerators and integral generators.

\begin{lemma}\label{miana=zers} (\cite{knjigah})
A closed linear operator $A$ is the generator of a \emph{(C-DS)} ${\mathcal G}$ iff
for every $\tau >0$ there exist an integer $n_{\tau}\in {\mathbb N}$
and a local $n_{\tau}$-times integrated $C$-semigroup
$(S_{n}(t))_{t\in \lbrack 0,\tau )}$ with the integral generator
$A.$  If this is the case, then the following equality holds:
$$
{\mathcal G}(\varphi)x=(-1)^{n}\int
\limits^{\tau}_{0}\varphi^{(n)}(t)S_{n}(t)x\, dt,\ x\in E,\ \varphi \in
{\mathcal D}_{(-\infty, \tau)}.
$$
\end{lemma}

Let us recall that the solution space for a closed linear operator $A$, denoted by $Z(A)$,
is defined as the set of all $x\in E$ for which there exists a continuous mapping $u(\cdot,x)\in C([0,\infty):E)$ satisfying $\int^t_0 u(s,x)\,ds\in D(A)$ and $A\int^t_0 u(s,x)\,ds=u(t,x)-x$, $t\geq 0$. Throughout the paper, we will use the following important characterization of space $Z(A):$
Assume $A$ generates a (C-DS) $\mathcal{G}$.
Denote by $D(\mathcal{G})$ the set of all $x\in\bigcap_{t\geq 0}D(G(\delta_t))$
satisfying that the mapping $t\mapsto G(\delta_t)x$, $t\geq 0$ is continuous.
Then $Z(A)=D(\mathcal{G})$.
If $x\in Z(A)$, then $u(t,x)=G(\delta_t)x$, $t\geq 0$ and
$
{\mathcal G}(\psi)x=\int^{\infty}_0\psi(t)Cu(t,x)\,dt,
\ \psi\in\mathcal{D}_0.
$

We need to recall the assertion of \cite[Proposition 3.1.28(ii)]{knjigah} for later purposes.

\begin{lemma}\label{555}
Assume that, for every $\tau>0$, there exists $n_{\tau}\in\mathbb{N}$
such that $A$ is a subgenerator of a local $n_{\tau}$-times integrated $C$-semigroup $(S_{{n_{\tau}}}(t))_{t\in[0,\tau)}$.
Then the solution space $Z(A)$ is the space which consists exactly of those elements $x\in E$ such that, for every $\tau>0$,
$S_{n_{\tau}}(t)x\in R(C)$ and that the mapping $t\mapsto C^{-1}S_{{n_{\tau}}}(t)x$, $t\in[0,\tau)$
is $n_{\tau}$-times continuously differentiable; if this is the case, then we have $G(\delta_t)x= (d^{n_{\tau}}/dt^{n_{\tau}})C^{-1}S_{{n_{\tau}}}(t)x$, $t\in[0,\tau).$
\end{lemma}

For any $x\in Z(A)$, the function $u(\cdot,x)\in C([0,\infty):E)$ satisfying the above requirements is said to be a mild solution of the abstract Cauchy problem
\[(ACP)_{1}:\left\{
\begin{array}{l}
u\in C([0,\infty):[D(A)])\cap C^1([0,\infty):E),\\
u^{\prime}(t)=Au(t),\;t\geq  0,\\
u(0)=x.
\end{array}
\right.
\]

Let $\zeta \in {\mathcal D}_{[-2,-1]}$ be a fixed test
function satisfying $\int_{-\infty }^{\infty }\zeta (t)\, dt=1.$
Then, with $\zeta$ chosen in this way, we define $I(\varphi)$
($\varphi \in {\mathcal D}$) as follows
$$
I(\varphi)(\cdot):=\int\limits_{-\infty }^{\cdot}\Biggl[\varphi (t)-\zeta
(t)\int_{-\infty }^{\infty }\varphi (u)\, du\Biggr]\, dt.
$$
Then $I(\varphi)\in {\mathcal D},$ $I(\varphi^{\prime})=\varphi ,$
$\frac{d}{dt}I(\varphi)(t)=\varphi(t)-\zeta(t)\int_{-\infty
}^{\infty }\varphi (u)\, du,$ $t\in \mathbb{R}$ and, for every $G\in
{\mathcal D}^{\prime }(L(E)),$ the primitive $G^{-1}$ of $G$ is defined
by setting $G^{-1}(\varphi ):=-G(I(\varphi)),\ \varphi \in {\mathcal
D}.$ It is clear that $G^{-1}\in {\mathcal D}^{\prime }(L(E)),$
$(G^{-1})^{\prime }=G,$ i.e.,
$-G^{-1}(\varphi^{\prime})=G(I(\varphi^{\prime}))=G(\varphi),\
\varphi \in {\mathcal D}$ and that supp$(G)\subseteq [0, \infty)$ implies
supp$(G^{-1})\subseteq [0, \infty).$

Now we recall definition of a $C$-distribution cosine function (cf. \cite{knjigaho}-\cite{FKP} for more details):

\begin{definition}\label{mesut}
An element ${\mathbf G}\in {\mathcal
D}_{0}^{\prime }(L(E))$ is called a pre$-(C-DCF)$ iff ${\mathbf
G}(\varphi)C=C{\mathbf G}(\varphi),$ $\varphi \in {\mathcal D}$ and
\[
(C-DCF_{1}) : {\mathbf G}^{-1}(\varphi \ast _{0}\psi )C={\mathbf
G}^{-1}(\varphi ){\mathbf G}(\psi )+{\mathbf G}(\varphi ){\mathbf
G}^{-1}(\psi ),\ \varphi ,\ \psi \in {\mathcal D};
\]
if, additionally,
\[
(C-DCF_{2}):\qquad x=y=0\;\mbox{iff}\;\text{\ }{\mathbf G}(\varphi
)x+{\mathbf G}^{-1}(\varphi )y=0,\ \varphi \in {\mathcal D}_{0},
\]
then ${\mathbf G}$ is called a $C$-distribution cosine function, in
short $(C-DCF).$ A pre$-(C-DCF)$ ${\mathbf G}$ is called dense if
the set ${\mathcal R}({\mathbf G}):=\bigcup_{\varphi \in {\mathcal
D}_{0}}\hbox{R}({\mathbf G}(\varphi))$ is dense in $E.$
\end{definition}

Notice that $(DCF_{2})$ implies $\bigcap_{\varphi \in {\mathcal D}%
_{0}}N({\mathbf G}(\varphi ))=\{0\}$ and $\bigcap_{\varphi \in {\mathcal D}%
_{0}}N({\mathbf G}^{-1}(\varphi ))=\{0\},$ and that the
assumption ${\mathbf G}\in {\mathcal D}_{0}^{\prime }(L(E))$ implies
${\mathbf G}(\varphi )=0,$ $\varphi \in {\mathcal D} _{(-\infty , 0]}.$

The (integral) generator
$A$ of ${\mathbf G}$ is defined by
\[
A:=\Bigl\{(x,y)\in E \times E : {\mathbf
G}^{-1}\bigl(\varphi^{\prime \prime}\bigr)x={\mathbf G}^{-1}(\varphi )y \mbox{
for all }\varphi \in {\mathcal D}_{0} \Bigr\}.
\]
It is well known that $({\mathbf G}(\psi )x$, ${\mathbf G}(\psi ^{\prime\prime})x+\psi ^{\prime}(0)Cx)\in A$, $%
\psi \in {\mathcal D}$, $x\in E$ and $({\mathbf G}^{-1}(\psi )x, -{\mathbf G}(\psi^{\prime})x-\psi (0)Cx)\in A$, $\psi
\in {\mathcal D}$, $x\in E$.

Let us remind ourselves of the following fundamental results (\cite{knjigaho}):

\begin{lemma}\label{lemm}
\begin{itemize}
\item[(i)]
Let ${\mathbf G}\in {\mathcal D}_{0}^{\prime }(L(E))$ and ${\mathbf G}(\varphi)C=C{\mathbf G}(\varphi),$
$\varphi \in {\mathcal D}.$ Then ${\mathbf G}$ is a (C-DCF) in $E$
generated by $A$ iff ${\mathcal G}$ is a \emph{(${\mathcal C}$-DS)} in $E\oplus
E$ generated by ${\mathcal A},$ where ${\mathcal A}\equiv \left(
\begin{array}{cc}
0 & I \\
A & 0
\end{array}
\right) ,$ ${\mathcal C}\equiv \left(
\begin{array}{cc}
C & 0 \\
0 & C
\end{array}
\right) $ and ${\mathcal
G}\equiv \left(
\begin{array}{cc}
{\mathbf G} & {\mathbf G}^{-1} \\
{\mathbf G}^{\prime }-\delta \otimes C & {\mathbf G}
\end{array}
\right) .$
\item[(ii)] A closed linear operator $A$ is the generator of a (C-DCF) ${\mathbf G}$ iff
for every $\tau >0$ there exist an integer $n_{\tau}\in {\mathbb N}$
and a local $n_{\tau}$-times integrated $C$-cosine function
$(C_{n_{\tau}}(t))_{t\in \lbrack 0,\tau )}$ with the integral generator
$A.$  If this is the case, then the following equality holds:
$$
{\mathbf G}(\varphi)x=(-1)^{n}\int
\limits^{\tau}_{0}\varphi^{(n)}(t)C_{n_{\tau}}(t)x\, dt,\ x\in E,\ \varphi \in
{\mathcal D}_{(-\infty, \tau)}.
$$
\end{itemize}
\end{lemma}

A $(C-DCF)$ ${\mathbf G}$ is said to be an
exponential $C$-distribution cosine function, $(E-CDCF)$ in short,
iff ${\mathcal G}$ is an $(E-{\mathcal C}DS)$ in $E \oplus E.$ The above is equivalent to say that there exists $\omega \in {\mathbb {R}}$ such that
$ e^{-\omega t}{\mathbf G}^{-1}\in {\mathcal S}^{\prime }(L(E)).$ A (C-DS) ${\mathcal G}$ is said to be an
exponential $C$-distribution semigroup, (E-CDS) in short, iff there exists $\omega \in {\mathbb {R}}$ such that
$ e^{-\omega t}{\mathcal G}\in {\mathcal S}^{\prime }(L(E)).$ It is well known that $A$ is the generator of an $(E-CDCF)$ in
$E$ (an (E-CDS) in
$E$) iff there exists $n \in {\mathbb {N}}$ such that  $A$ is the integral generator of an exponentially bounded $n$-times integrated
$C$-cosine function  in
$E$ (an exponentially bounded $n$-times integrated
$C$-semigroup  in
$E$).

A function $u(\cdot;x,y)$ is said to be a mild solution of the
abstract Cauchy problem
\[(ACP)_{2}:\left\{
\begin{array}{l}
u\in C([0,\infty):[D(A)])\cap C^2([0,\infty):E),\\
u^{\prime \prime}(t)=Au(t),\;t\geq  0,\\
u(0)=x,\;u^{\prime}(0)=y
\end{array}
\right.
\]
iff the mapping the mapping
$t\mapsto u(t;x,y),\ t\geq 0$ is continuous, $\int
^{t}_{0}(t-s)u(s;x,y)ds\in D(A)$ and $A\int
^{t}_{0}(t-s)u(s;x,y)ds=u(t;x,y)-x-ty,$ $t\geq 0;$ in the sequel, we
primarily consider the mild solutions of $
(ACP)_{2}$ with $y=0.$ Denote by $Z_{2}(A)$ the set which consists of all $x\in
E$ for which there exists such a solution. Let $\pi_{1} : E\times E
\rightarrow E$ and $\pi_{2} : E\times E \rightarrow E$ be the
projections and let $G$ be a $(C-DCF)$ generated by $A.$ Then
${\mathcal G}$ is a (${\mathcal C}$-DS) generated by ${\mathcal A}$
and we define
$G(\delta_{t})x:=\pi_{2}({{\mathcal G}}(\delta
_{t})\binom{0}{x}),$ $t\geq 0,$ $x\in Z_{2}(A).$ Let us recall that, for every $x\in Z_{2}(A),$ one has
$G(\delta_{t})(Z_{2}(A))\subseteq Z_{2}(A),$ $t\geq 0,$ $
2G(\delta_{s})G(\delta_{t})x=G(\delta_{t+s})x+G(\delta_{|t-s|})x,$
$t,\ s\geq 0$ and ${\mathbf G}(\varphi)x=\int_{0}^{\infty}\varphi(t)CG(\delta_{t})x\, dt,$ $\varphi \in {{\mathcal D}_{0}}.$

\begin{lemma}\label{tadic} (\cite{knjigaho})
Assume that, for
every $\tau
>0,$ there exists $n_{\tau} \in {\mathbb N}$ such that $A$ is a
subgenerator of a local $n_{\tau}$-times integrated $C$-cosine
function $(C_{n_{\tau}}(t))_{t\in [0,\tau)}.$ Then the solution
space $Z_{2}(A)$ consists exactly of those vectors $x\in E$ such
that, for every $\tau
>0,$ $C_{n_{\tau}}(t)x\in \hbox{R}(C)$ and that the mapping
$t\mapsto C^{-1}C_{n_{\tau}}(t)x,$ $t\in [0,\tau)$ is
$n_{\tau}$-times continuously differentiable. If $x\in Z_{2}(A)$ and
$t\in [0,\tau),$ then
$G(\delta_t)x= (d^{n_{\tau}}/dt^{n_{\tau}})C^{-1}C_{n_{\tau}}(t)x.$
\end{lemma}

Finally, we give a brief overview of the basic properties of almost periodic functions with values in Banach spaces.
Let $I={\mathbb R}$ or $I=[0,\infty),$ and let $f : I \rightarrow E.$ Given $\epsilon>0,$ we call $\tau \in I$ an $\epsilon$-period for $f(\cdot)$ iff
\begin{align*}
\| f(t+\tau)-f(t) \| \leq \epsilon,\quad t\in I.
\end{align*}
The set constituted of all $\epsilon$-periods for $f(\cdot)$ is denoted by $\vartheta(f,\epsilon).$ It is said that $f(\cdot)$ is almost periodic, a.p. for short, iff for each $\epsilon>0$ the set $\vartheta(f,\epsilon)$ is relatively dense in $I,$ which means that
there exists $l>0$ such that any subinterval of $I$ of length $l$ meets $\vartheta(f,\epsilon).$ We call $f(\cdot)$
weakly almost periodic, w.a.p. for short, iff for each $x^{\ast}\in E^{\ast}$ the function $x^{\ast}(f(\cdot))$ is almost periodic.

By $AP(I:E)$ we denote the vector space consisting of all almost periodic functions from the interval $I$ into $E.$ Equipped with the sup-norm, $AP(I:E)$ becomes a Banach space.

The concept of almost periodicity was first studied by H. Bohr in 1925 and later generalized by many other mathematicians.
Almost periodic Banach space valued functions has been investigated in \cite{18} and \cite{30}; we can also recommend reading the monographs \cite{gaston}-\cite{gue} by G. M. N'Gu\' er\' ekata and \cite{diagana} by T. Diagana.

The most intriguing properties of almost periodic vector-valued functions are collected in the following theorem, stated here for the sake of clarity and better understanding of material which is to follow.

\begin{theorem}\label{svinja}
Let $f \in AP({\mathbb R} : E).$ Then the following holds:
\begin{itemize}
\item[(i)] $f(t)$ is bounded, i.e., $\sup_{t\in {\mathbb R}} \|f(t)\| <\infty ;$
\item[(ii)] if $g \in AP({\mathbb R} :E),$ $h \in AP({\mathbb R} :{\mathbb C}),$ then $f + g$ and $ hf \in AP({\mathbb R} :E);$
\item[(iii)] $P_{r}(f) := \lim_{t\rightarrow \infty}\frac{1}{t}\int^{t}_{0}e^{-irs}f(s)\, ds$ exists for all $r\in {\mathbb R}$ (Bohr's transform of $f(\cdot)$) and
$P_{r}(f) := \lim_{t\rightarrow \infty}\frac{1}{t}\int^{t+\alpha}_{\alpha}e^{-irs}f(s)\, ds$ for all $\alpha,\ r\in {\mathbb R};$
\item[(iv)] if $P_{r}(f) = 0$ for all $r \in {\mathbb R},$ then $f(t) = 0$ for all $t \in {\mathbb R};$
\item[(v)] $\sigma(f):=\{r\in {\mathbb R} :  P_{r}(f) \neq 0\}$ is at most countable;
\item[(vi)] if $c_{0} \nsubseteq E,$ which means that $E$ does not contain an isomorphic copy of $c_{0},$
and $g(t) =\int^{t}_{0}f(s)\, ds$ ($t \in {\mathbb R}$) is bounded, then $g \in AP({\mathbb R} :E);$
\item[(vii)] if $(g_{n})_{n\in {\mathbb N}}$ is a sequence in $AP({\mathbb R}:E)$ and $(g_{n})_{n\in {\mathbb N}}$ converges uniformly to $g$, then
$g \in AP({\mathbb R}:E);$
\item[(viii)] if $f^{\prime} \in BUC({\mathbb R}:E),$ then $f^{\prime} \in AP({\mathbb R} : E).$
\end{itemize}
\end{theorem}

For the sequel, we need some preliminaries from the pioneering paper \cite{bart} by H. Bart and S. Goldberg. The translation semigroup $(W(t))_{t\geq 0}$ on $AP([0,\infty) : E),$ given by $[W(t)f](s):=f(t+s),$ $t\geq 0,$ $s\geq 0,$ $f\in AP([0,\infty) : E)$ is consisted solely of surjective isometries $W(t)$ ($t\geq 0$) and can be extended to a $C_{0}$-group $(W(t))_{t\in {\mathbb R}}$ of isometries on $AP([0,\infty) : E),$ where $W(-t):=W(t)^{-1}$ for $t>0.$ Furthermore, the mapping $F : AP([0,\infty) : E) \rightarrow AP({\mathbb R} : E),$ defined by
$$
[Ff](t):=[W(t)f](0),\quad t\in {\mathbb R},\ f\in AP([0,\infty) : E),
$$
is a linear surjective isometry and $Ff$ is the unique continuous almost periodic extension of a function $f$ from $AP([0,\infty) : E)$ to the whole real line. We have that $
[F(Bf)]=B(Ff)$ for all $B\in L(E)$ and $f\in  AP([0,\infty) : E).$

We refer the reader to the monograph \cite{hino-bor} by
Y. Hino, T. Naito, N. V. Minh and J. S. Shin for further information concerning
almost periodic solutions of abstract differential equations in Banach spaces.

\section{Formulation and proof of main results}

We start this section by introducing the following definition.

\begin{definition}\label{ceo-er}
Let ${\mathbf G}$ be a $(C-DCF)$ generated by $A,$ resp. let ${\mathcal G}$ be a \emph{(C-DS)} generated by $A$. Suppose that $\tilde{E}$
is a linear subspace of $Z_{2}(A),$ resp. $x\in Z(A).$
Then it is said that ${\mathbf G}$ is $\tilde{E}$-(weakly) almost periodic iff for each
$x\in \tilde{E}$ the mapping $t\mapsto G(\delta_{t})x,$
$t\geq 0$ is (weakly) almost periodic.
\end{definition}

\begin{remark}\label{ceo-er}
\begin{itemize}
\item[(i)]
It is clear that the above notions can be introduced for arbitrary operator family $(F(t))_{t\geq 0}$ consisted of possibly non-linear and possibly non-continuous single valued operators.
\item[(ii)] Suppose that ${\mathbf G}$ is a $(C-DCF)$ generated by $A,$ resp. ${\mathcal G}$ is a \emph{(C-DS)} generated by $A$. Let $\tilde{E}$
be a linear subspace of $Z_{2}(A),$ resp. $x\in Z(A),$ such that
${\mathbf G},$ resp. ${\mathcal G},$ is $\tilde{E}$-(weakly) almost periodic. Let ${\mathbf G}_{1}$ be a $(C_{1}-DCF)$ generated by $A,$ resp. let ${\mathcal G}_{1}$ be a \emph{(C$_{1}$-DS)} generated by $A$.
Then ${\mathbf G}_{1},$ resp. ${\mathcal G}_{1},$ is $\tilde{E}$-(weakly) almost periodic.
\end{itemize}
\end{remark}

Let $A$ be a closed linear operator. Designate by
$D$ ($H$) the set consisting of all eigenvectors of operator $A$ which correspond to purely imaginary eigenvalues of operator $A$ (to non-positive real eigenvalues of operator $A$); we assume henceforth that
the set $D_{0}$ ($H_{0}$) is consisted of all eigenvectors of operator $A$ which correspond to purely imaginary non-zero eigenvalues of operator $A$ (to negative real eigenvalues of operator $A$).

Motivated by our considerations from the point 10. of \cite{aot}, we state and prove the following result.

\begin{proposition}\label{spremte-se}
\begin{itemize}
\item[(i)] Suppose that ${\mathcal G}$ is a \emph{(C-DS)} generated by $A.$ Then $\tilde{E}:=span(D) \subseteq Z(A)$ and
${\mathcal G}$ is $\tilde{E}$-almost periodic. Furthermore,
the mapping $t\mapsto \int^{t}_{0}G(\delta_{s})x\, ds,$ $t\geq 0$ is almost periodic for all
$x\in span(D_{0}).$
\item[(ii)] Suppose that ${\mathbf G}$ is a $(C-DCF)$ generated by $A.$ Then $\tilde{E}:=span(H) \subseteq Z_{2}(A)$ and
${\mathbf G}$ is $\tilde{E}$-almost periodic. Furthermore,
the mappings $t\mapsto \int^{t}_{0}G(\delta_{s})x\, ds,$ $t\geq 0$ and $t\mapsto \int^{t}_{0}(t-s)G(\delta_{s})x\, ds,$ $t\geq 0$ are almost periodic for all
$x\in span(H_{0}).$
\end{itemize}
\end{proposition}

\begin{proof}
We will prove only (i) and outline some basic facts needed for the proof of (ii).
By  Lemma \ref{miana=zers}, we know that, for every $\tau >0,$ there exist an integer $n=n_{\tau}\in {\mathbb N}$
and a local $n$-times integrated $C$-semigroup
$(S_{n}(t))_{t\in [ 0,\tau )}$ with the integral generator
$A.$
Suppose that $r\in {\mathbb R}$
and $irx= Ax.$ Then $S_{n}(t)x-g_{n+1}(t)Cx=ir\int^{t}_{0}S_{n}(s)x\, ds,$ $t\in [ 0,\tau ),$ which simply implies that the mapping $t\mapsto S_{n}(t)x,$ $t\in [ 0,\tau )$
is infinitely differentiable with all derivatives at zero of order less than or equal to $n-1$ being zeroes and that $(d^{n}/dt^{n})S_{n}(t)x=e^{irt}Cx,$ $t\in [ 0,\tau ).$ Hence, $S_{n}(t)x=(g_{n}\ast_{0}e^{ir\cdot})Cx,$ $t\in [ 0,\tau ),$
$G(\delta_{t})x=e^{irt}x,$ $t\in [ 0,\tau )$ and the last equality clearly continues to hold for all non-negative reals $t.$
Now the final conclusions follow from Lemma \ref{555}.
The proof of part (ii) is quite similar, and can be deduced by using Lemma \ref{lemm}(ii), Lemma \ref{tadic}
and
the equality $G(\delta_{t})x=\cos(rt)x,$ $t\geq 0,$ provided that $-r^{2}x=Ax$ for some $r\in {\mathbb R}.$
\end{proof}

With the help of Proposition \ref{spremte-se}, we can simply construct a great number of non-exponential $C$-distribution semigroups ($C$-distribution cosine functions) that are $\tilde{E}$-almost periodic, with the subspace $\tilde{E}$ being dense in $E$ (see e.g. the extension of Desch-Schappacher-Webb criterion for chaos of strongly continuous semigroups \cite[Theorem 3.1.36]{knjigah}, \cite[Theorem 2.2.10]{knjigaho} and \cite[Example 3.1.41]{knjigah}). As mentioned in the introductory part, our recent research studies of hypercyclic and topologically mixing properties of strongly continuous semigroups and cosine functions (see \cite[Chapter III]{knjigaho} for a comprehensive survey of results) enable one to construct
many other examples of subspace almost periodic strongly continuous semigroups and cosine functions that are not almost periodic in the usual sense (\cite{aot}).

We want also to mention that Proposition \ref{spremte-se} continues to hold for $C$-distribution semigroups and $C$-distribution cosine functions in locally convex spaces, which follows from the fact that the equalities
$G(\delta_{t})x=e^{irt}x,$ $t\geq 0$ and $G(\delta_{t})x=\cos(rt)x,$ $t\geq 0$ used above can be proved without assuming that the vector-valued distributions ${\mathcal G}$ and ${\mathbf G}$ are of finite order
(cf. \cite{FKP} for the notion and \cite[Lemma 2.8]{chaotic} for semigroup case; for cosine operator case, combine the above result with Lemma \ref{lemm}(i), and \cite[Lemma 3.2.33]{knjigaho} with $\lambda=\pm ir$).

Suppose that ${\mathcal G}$ is a (C-DS) generated by $A,$ and ${\mathcal G}$ is $\tilde{E}$-almost periodic.
Define $T(t)x:=T_{x}(t):=G(\delta_{t})x,$
$t\geq 0,$ $x\in Z(A)$
and
$S(t)x:=[F(T_{x}(\cdot))](t),$ $t\in {\mathbb R},$ $x\in \tilde{E}.$ Since $F$ is a linear surjective isometry between the spaces $AP([0,\infty) : E)$ and $AP({\mathbb R} : E),$ we have that
$$
\|S(t)x\|\leq \sup_{s\in {\mathbb R}}\|S(s)x\|=\sup_{s\geq 0}\|S(s)x\|=\sup_{s\geq 0}\|T(s)x\|,\quad x\in \tilde{E},\ t<0,
$$
and therefore,
$$
\sup_{t\in {\mathbb R}}\|S(t)x\|=\sup_{t\geq 0}\|S(t)x\|,\quad x\in \tilde{E}.
$$
Furthermore, it can be easily seen that $S(\cdot)$ commutes with $C.$

\begin{proposition}\label{polugrupe}
We have the following:
\begin{itemize}
\item[(i)] $T(t)T(s)x=T(t+s)x$
for all $t\geq 0,$ $ s\geq 0$ and $ x\in Z(A).$
\item[(ii)] Suppose that $t\geq 0,$ $ s\leq 0,$ $ x\in \tilde{E}$ and $ G(\delta_{t})x\in \tilde{E}.$ Then
$S(s)S(t)x=S(t+s)x.$
\item[(iii)] Suppose that $t\geq 0,$ $ s\leq 0$ and $ x\in \tilde{E}.$ Then $S(s)x\in Z(A)
$ and $T(t)S(s)x=S(t+s)x.
$
\item[(iv)] Suppose that $t\geq 0,$ $ s\leq 0,$ $ x\in \tilde{E},$ $ G(\delta_{r})x\in \tilde{E}$ for all $r\geq  0$ and $\tilde{E}$ is closed. Then
$S(s)x\in \tilde{E}
$ and $S(t)S(s)x=S(t+s)x.
$
\item[(v)] Suppose that $t\leq 0,$ $ s\leq 0,$ $ x\in \tilde{E},$ $ G(\delta_{r})x\in \tilde{E}$ for all $r\geq  0$ and $\tilde{E}$ is closed. Then
$S(t)S(s)x=S(t+s)x.
$
\end{itemize}
\end{proposition}

\begin{proof}
The part (i) follows immediately from A1., while the proof of (ii) can be deduced by using definition and properties of extension mapping $F,$ A1. and the inclusion $ G(\delta_{t})x\in \tilde{E}:$
\begin{align*}
S(s)& S(t)x=S(s)G(\delta_{t})x
\\&=\bigl[ F(T_{G(\delta_{t})x}(\cdot)) \bigr](s)
=\bigl[ W(s)T_{G(\delta_{t})x}(\cdot) \bigr](0)
=\bigl[ W(s)W(t)T_{x}(\cdot) \bigr](0)
\\ & =\bigl[W(t+s)T_{x}(\cdot)  \bigr](0)
=\bigl[ F(T_{x}(\cdot))\bigr](t+s)=S(t+s)x.
\end{align*}
For (iii), we need only to prove that
\begin{align*}
A\int^{t}_{0}u(r;S(s)x)\, dr=u(t;S(s)x)-S(s)x,
\end{align*}
where $u(r;S(s)x):=S(r+s)x,$ $r\geq 0.$ It is checked at once that the mapping
$r\mapsto u(r;S(s)x),$ $r\geq 0$ is continuous and $u(0;S(s)x)=S(s)x.$
Let us suppose first that $t\geq -s .$ By definition of integral generator $A$, we need to prove that:
\begin{align}\label{steva}
{\mathcal G}\bigl(-\varphi^{\prime}\bigr)\int^{t}_{0}u(r;S(s)x)\, dr={\mathcal G}
(\varphi)[u(t;S(s)x)-S(s)x],\quad \varphi \in {\mathcal D}_{0}.
\end{align}
But, we have
\begin{align*}
{\mathcal G}\bigl(-& \varphi^{\prime}\bigr)\int^{t}_{0}u(r;S(s)x)\, dr
\\ &= {\mathcal G}\bigl(- \varphi^{\prime}\bigr)\Biggl(\int^{-s}_{0}+\int^{t}_{-s}\Biggr)u(r;S(s)x)\, dr
\\ &={\mathcal G}\bigl(- \varphi^{\prime}\bigr)\int^{-s}_{0}\bigl[ F(T_{x}(\cdot)) \bigr](r+s)\, dr + {\mathcal G}\bigl(- \varphi^{\prime}\bigr)\int^{t}_{-s}G(\delta_{r+s})x\, dr
\\ &={\mathcal G}\bigl(- \varphi^{\prime}\bigr)\int^{0}_{s}\bigl[ F(T_{x}(\cdot)) \bigr](r)\, dr +{\mathcal G}\bigl(- \varphi^{\prime}\bigr)\int^{t+s}_{0}G(\delta_{r})x\, dr
\\ &={\mathcal G}\bigl(- \varphi^{\prime}\bigr)\int^{0}_{s}S(r)x\, dr +{\mathcal G}(\varphi)[u(t+s;x)-x],\quad \varphi \in {\mathcal D}_{0},
\end{align*}
where the last equality follows from definitions of integral generator $A$ and mild solution $u(\cdot;x).$ Hence, \eqref{steva} is equivalent with
\begin{align}\label{steva-prim}
{\mathcal G}(\varphi)S(s)x-{\mathcal G}(\varphi)x
={\mathcal G}\bigl( \varphi^{\prime}\bigr)\int^{0}_{s}S(r)x\, dr,\quad \varphi \in {\mathcal D}_{0},
\end{align}
because $u(t;S(s)x)=G(\delta_{t+s})x=u(t+s;x). $
For this, fix a test function $\varphi \in {\mathcal D}_{0}$ and put $f(v):={\mathcal G}(\varphi)S(v)x,$ $v\leq 0.$ By the Newton-Leibniz formula, it suffices to show that
\begin{align}\label{jankovici}
f^{\prime}(v)={\mathcal G}\bigl(- \varphi^{\prime}\bigr)S(v)x,\quad v\leq 0.
\end{align}
Observe first that the property A1.-A2. along with definition of mild solution $u(\cdot;x)$ imply:
\begin{align*}
{\mathcal G}(\varphi)T_{x}(\cdot +\sigma)- {\mathcal G}(\varphi)T_{x}(\cdot)=\int^{t}_{0}{\mathcal G}\bigl(- \varphi^{\prime}\bigr)T_{x}(\cdot +\sigma)\, d\sigma,\quad \sigma \geq 0,
\end{align*}
i.e.,
\begin{align*}
W(\sigma)[{\mathcal G}(\varphi)T_{x}(\cdot)]-{\mathcal G}(\varphi)T_{x}(\cdot)=\int^{\sigma}_{0}W(s')\bigl[{\mathcal G}\bigl(- \varphi^{\prime}\bigr)T_{x}(\cdot)\bigr]\, ds',\quad \sigma \geq 0.
\end{align*}
Because of that, the pair $({\mathcal G}(\varphi)T_{x}(\cdot),{\mathcal G}(- \varphi^{\prime})T_{x}(\cdot))$ belongs to the graph of the  integral generator of $(W(t))_{t\geq 0}$ on $AP([0,\infty) : E),$
which is well known that equals to the infinitesimal generator of $(W(t))_{t\geq 0}$
($(W(t))_{t\in {\mathbb R}}$). Hence,
\begin{align}\label{bnm}
\lim_{h\rightarrow 0}\frac{W(h)-I}{h}{\mathcal G}(\varphi)T_{x}(\cdot)={\mathcal G}\bigl(- \varphi^{\prime}\bigr)T_{x}(\cdot),
\end{align}
for the topology of $AP([0,\infty) : E).$
Using \eqref{bnm}, the group property of $(W(t))_{t\in {\mathbb R}},$ the fact that $F$ is a linear surjective isometry, and elementary definitions, we get that
\begin{align*}
f^{\prime}(v)&=\lim_{h\rightarrow 0}\Biggl[W(v)\frac{W(h)-I}{h}{\mathcal G}(\varphi)T_{x}(\cdot)\Biggr](0)
\\&=\Bigl[W(v){\mathcal G}\bigl(- \varphi^{\prime}\bigr)T_{x}(\cdot)\Bigr](0)
=\Bigl[F\Bigl({\mathcal G}\bigl(- \varphi^{\prime}\bigr)T_{x}(\cdot)\Bigr)\Bigr](v)
\\&={\mathcal G}\bigl(- \varphi^{\prime}\bigr)[F(T_{x}(\cdot))](v)
={\mathcal G}\bigl(- \varphi^{\prime}\bigr)S(v)x,\quad v\leq 0,
\end{align*}
completing the proof of \eqref{jankovici} and (iii) in the case that $t\geq -s .$ The case $ t\leq -s$ is much easier and follows by applying \eqref{steva-prim} and elementary substitutions in integrals.
In order to prove (iv), it suffices to show, by definition of an $E$-valued almost periodic function, that for each $n\in {\mathbb N}$ there exists a positive number $l_{n}$ such that any interval $I$ of length $l_{n}$ contains a number $\tau_{n}$ such that
$
\| S(v+\tau_{n})x-S(v)x \| \leq 1/n,\ v\in {\mathbb R}
$ and, in particular, $
\| S(s+\tau_{n})x-S(s)x \| \leq 1/n.
$ Taking $
I \subseteq [-s+1,\infty)
$ we obtain that $S(s)x\in \overline{\{S(r)x: r\geq 0\}}.$
Since $S(r)x= G(\delta_{r})x\in \tilde{E}$ for all $r\geq  0$ and $\tilde{E}$ is closed, (iv) immediately follows from the above inclusion. To prove (v), observe that (iv) implies
$S(v)x\in \tilde{E}
$ for $v\leq 0.$ Plugging $y=S(s)x,$ the equality in (v) is equivalent  with
$S(t)y=S(t+s)S(-s)y,$ which follows from an application of (ii), where we use the equality
$T(-s)S(s)y=y,$ which is true due to (iii).
\end{proof}

\begin{remark}\label{srem}
The parts (iv) and (v) hold in the case that $\tilde{E} = Z(A),$ no matter whether this subspace is closed or not in $E.$ In general case, it is not clear whether the parts
(iv) and (v) hold provided that
$\tilde{E} \neq Z(A)$ and $\tilde{E}$ is not closed in $E.$
\end{remark}

\begin{remark}\label{srem-prim}
Suppose that $\tilde{E} = Z(A),$ and $Z(A)$ is dense in $E.$ Due to Lemma \ref{miana=zers}, for every $\tau >0,$ there exist an integer $n_{\tau}\in {\mathbb N}$
and a local $n_{\tau}$-times integrated $C$-semigroup
$(S_{n_{\tau}}^{+}(t))_{t\in [0,\tau )}$ with the integral generator
$A.$ Define
\begin{align*}
S_{n_{\tau}}^{-}(t)x:=\int^{t}_{0}\frac{(t-s)^{n_{\tau}-1}}{(n_{\tau}-1)!}CS(-s)x\, ds,\quad t\in [0,\tau ),\ x\in Z(A).
\end{align*}
Then it is not clear how one can prove that there exists a finite constant $M_{\tau}$ such that
$\|S_{n_{\tau}}^{-}(t)x\| \leq M_{\tau}\|x\|,$ $t\in [0,\tau ),$ $x\in Z(A).$ Because of that, we are not able to extend $ S_{n_{\tau}}^{-}(t)$ to a bounded linear operator ${\bf S}_{n_{\tau}}^{-}(t)$ acting on $ E,$  despite of our assumption made on density of $Z(A)$ in $E.$
It is clear that, owing to Proposition \ref{polugrupe}, we can expect from $({\bf S}_{n_{\tau}}^{-}(t))_{t\in [0,\tau)}$ to be a local $n_{\tau}$-times integrated $C$-semigroup
generated by
$-A.$ Of course, the validity of the last statement for all numbers $\tau>0$ would imply by Lemma \ref{555} that
$Z(A)\subseteq Z(-A)$ and that
$-A$ generates a  $\tilde{E}$-almost periodic $C$-distribution semigroup in $E.$
\end{remark}

In connection with Remark \ref{srem-prim}, we have the following comments to make:
\begin{itemize}
\item[1.]
Suppose that $A$ generates a global $C$-regularized semigroup $(Q_{+}(t))_{t\geq 0}.$ Denote by ${\mathcal G}_{+}$ the induced $C$-distribution semigroup generated by $A.$ Then
it can be easily seen that $(Q_{+}(t))_{t\geq 0}$ is almost periodic in the sense of \cite[Definition 1.1]{aot} (i.e., for any element $x\in E$
the mapping $t\mapsto Q_{+}(t)x,$ $t\geq 0$ is almost periodic) iff ${\mathcal G}_{+}$ is $R(C)$-almost periodic. If this is the case, \cite[Theorem 3.2]{aot} implies that the operator
$-A$ generates an almost periodic $C$-regularized semigroup $(Q_{-}(t))_{t\geq 0},$ and therefore, $-A$ generates a  $C$-distribution semigroup ${\mathcal G}_{-}$ that is $R(C)$-almost periodic.
\item[2.]
The situation is quite similar if we consider integrated $C$-semigroups. To explain this,
suppose that $A$ generates an exponentially bounded $n$-times integrated $C$-semigroup $(Q_{n,+}(t))_{t\geq 0}$ for some $n\in {\mathbb N}.$
Then \cite[Proposition 2.3.13]{knjigah} shows that there exists a positive real number $\lambda$ such that $A$ generates an exponentially bounded $((\lambda-A)^{-n}C)$-regularized semigroup $(R_{n,+}(t))_{t\geq 0},$ so that $R((\lambda-A)^{-n}C)\subseteq Z(A).$
Denote by ${\mathcal G}_{+}$ the induced exponential $C$-distribution semigroup generated by $A,$ and suppose that ${\mathcal G}_{+}$ is $R((\lambda-A)^{-n}C)$-almost periodic. Using the analysis from the above paragraph, the afore-mentioned proposition and
Remark \ref{ceo-er}(ii), it is very simple to prove that $-A$ generates an exponential $R((\lambda-A)^{-n}C)$-almost periodic $C$-distribution semigroup. The same conclusion can be formulated for weak almost periodicity.
\item[3.]
Here we would like to propose an interesting question with regard to almost periodicity of $C$-distribution cosine functions. Suppose that ${\mathbf G}$ is a $(C-DCF)$ generated by $A,$ and ${\mathbf G}$ is $\tilde{E}$-almost periodic.
Define $C(t)x:=C_{x}(t):=G(\delta_{t})x,$
$t\geq 0,$ $x\in Z_{2}(A)$
and
${\bf C}(t)x:=[F(C_{x}(\cdot))](t),$ $t\in {\mathbb R},$ $x\in \tilde{E}.$ As in the semigroup case, we have that ${\bf C}(\cdot)$ commutes with $C$ as well as that
$
\|{\bf C}(t)x\|\leq \sup_{s\geq 0}\|C(s)x\|,\ x\in \tilde{E},\ t<0,
$
and
$
\sup_{t\in {\mathbb R}}\|{\bf C}(t)x\|=\sup_{t\geq 0}\|C(t)x\|,\ x\in \tilde{E}.
$ In the present situation, we do not know to tell what conditions ensure the validity of
expected equality ${\bf C}(-t)x=C(t)x,$ for $t\geq 0$ and $x\in \tilde{E}.$
\end{itemize}

We continue by stating the following theorem.

\begin{theorem}\label{prcko}
\begin{itemize}
\item[(i)] Suppose that ${\mathcal G}$ is a \emph{(C-DS)} generated by $A,$ and ${\mathcal G}$ is $\tilde{E}$-almost periodic for some linear subspace $\tilde{E}$ of $Z(A)$. Put
\begin{align}\label{brod}
P_{r}x := \lim_{t\rightarrow \infty}\frac{1}{t}\int^{t}_{0}e^{-irs}G\bigl(\delta_{s}\bigr)x\, ds,\quad r\in {\mathbb R},\ x\in \tilde{E}.
\end{align}
Then $AP_{r}x=ir P_{r}x,$ $ r\in {\mathbb R},$ $x\in \tilde{E}.$ Furthermore, if $\tilde{E}$ is dense in $E,$ then the set $D$ consisting of all eigenvectors of operator $A$ which corresponds to purely imaginary eigenvalues of operator $A$ is total in $E$ (i.e., the linear span of $D$ is dense in $E$).
\item[(ii)] Suppose that ${\mathbf G}$ is a $(C-DCF)$ generated by $A,$ and ${\mathbf G}$ is $\tilde{E}$-almost periodic  for some linear subspace $\tilde{E}$ of $Z_{2}(A)$. Define, for every $r\in {\mathbb R}$ and $x\in \tilde{E},$ the element $P_{r}x$ through \eqref{brod}. Then
$AP_{r}x=-r^{2}P_{r}x,$ $ r\in {\mathbb R},$ $x\in \tilde{E}.$ Furthermore, if $\tilde{E}$ is dense in $E,$ then the set $H$ consisted of all eigenvectors of $A$ which corresponds to the real non-positive eigenvalues of $A$ is total in $E.$
\end{itemize}
\end{theorem}

\begin{proof}
Let $x\in \tilde{E}$ and $ r\in {\mathbb R}$ be fixed. We will only prove that $AP_{r}x=ir P_{r}x$ in semigroup case and that $AP_{r}x=-r^{2}P_{r}x,$ in cosine operator function case; then the remaining part of proof of (i)-(ii) follows by copying the final part of proof of \cite[Theorem
2.1]{zhengliu}. Consider first the semigroup case. Then we know that
 \begin{align*}
A\int^{t}_{0}u(s;x)\, ds=u(t;x)-x,\quad t\geq 0,\ x\in \tilde{E},
\end{align*}
i.e., that
\begin{align*}
{\mathcal G}\bigl(-& \varphi^{\prime}\bigr)\int^{t}_{0}G\bigl(\delta_{s}\bigr)x\, ds={\mathcal G}(\varphi)\Bigl[G\bigl(\delta_{t}\bigr)x-x\Bigr],\quad t\geq 0,\ x\in \tilde{E}, \ \varphi \in {\mathcal D}_{0}.
\end{align*}
Applying the partial integration, we get from the above that
\begin{align*}
{\mathcal G}\bigl(-& \varphi^{\prime}\bigr)P_{r}x
\\ &={\mathcal G}\bigl(- \varphi^{\prime}\bigr)\lim_{t\rightarrow \infty}\frac{1}{t}\Biggl[ e^{-irt}\int^{t}_{0}G\bigl(\delta_{s}\bigr)x\, ds+ir  \int^{t}_{0}e^{-irs}\int^{s}_{0}G\bigl(\delta_{v}\bigr)x\, dv \, ds \Biggr]
\\ & =\lim_{t\rightarrow \infty}\frac{1}{t}\Biggl[ e^{-irt}{\mathcal G}(\varphi)\Bigl \{G\bigl(\delta_{t}\bigr)x-x\Bigr\} +ir \int^{t}_{0}e^{-irs} {\mathcal G}(\varphi)\Bigl\{G\bigl(\delta_{s}\bigr)x-x\Bigr \}\, ds \Biggr],
\end{align*}
which can be simply shown to equals
$ir \int^{t}_{0}e^{-irs} {\mathcal G}(\varphi)G(\delta_{s})x\, ds $
by the boundedness of function $G\bigl(\delta_{\cdot}\bigr)x.$ This implies by definition of integral generator $A$ that $AP_{r}x=ir P_{r}x,$ as claimed. Consider now the cosine operator function case.
Then we have that \cite{knjigaho}:
 \begin{align*}
A\int^{t}_{0}(t-s)G\bigl(\delta_{s}\bigr)x\, ds=G\bigl(\delta_{t}\bigr)x-x,\quad t\geq 0,\ x\in \tilde{E},
\end{align*}
i.e., that
\begin{align}\label{noviw}
{\mathbf G}\bigl(& \varphi^{\prime}\bigr)\int^{t}_{0}(t-s)G\bigl(\delta_{s}\bigr)x\, ds={\mathbf G}(I(\varphi))\Bigl[G\bigl(\delta_{t}\bigr)x-x\Bigr],\quad t\geq 0,\ x\in \tilde{E}, \ \varphi \in {\mathcal D}_{0}.
\end{align}
Fix an element $x\in \tilde{E}$ and a test function $\varphi \in {\mathcal D}_{0},$ and define after that the function
$H(t):={\mathbf G}( \varphi^{\prime})\int^{t}_{0}(t-s)G\bigl(\delta_{s}\bigr)x\, ds,$ $t\geq 0.$ Then \eqref{noviw} in combination with the almost periodicity of function $G\bigl(\delta_{\cdot}\bigr)x$ implies that
$H(t)$ is two times continuously differentiable, with $H(t)$ and $H^{\prime \prime}(t)$ being bounded for $t\geq 0.$ Then the Landau inequality
\begin{align*}
\Bigl| \bigl \langle x^{\ast} , H^{\prime}(t) \bigr \rangle \Bigr|^{2}\leq 4  \Bigl| \bigl \langle x^{\ast} , H(t) \bigr \rangle \Bigr| \cdot \Bigl| \bigl \langle x^{\ast} , H^{\prime \prime}(t) \bigr \rangle \Bigr|,\quad t\geq 0,\ x^{\ast} \in E^{\ast},
\end{align*}
shows that the function $t\mapsto H^{\prime}(t),$ $t\geq 0$ is weakly bounded and therefore bounded. Hence, there exists a finite constant $M>0$ such that
\begin{align}\label{bvx}
\bigl\| H^{\prime}(t) \bigr\| = \Biggl\| {\mathbf G}\bigl( \varphi^{\prime}\bigr)\int^{t}_{0}G\bigl(\delta_{s}\bigr)x\, ds \Biggr\|   \leq M,\quad t\geq 0.
\end{align}
On the other hand,
applying the partial integration twice, we get that:
\begin{align}
\notag \int^{t}_{0}e^{-irs}G\bigl( \delta_{s}\bigr)x\, ds & =e^{-irt}\int^{t}_{0}G\bigl(\delta_{s}\bigr)x\, ds +ire^{-irt}\int^{t}_{0}(t-s)G\bigl(\delta_{s}\bigr)x\, ds
\\\label{stari-novi} & -r^{2}\int^{t}_{0}e^{-irs}\int^{s}_{0}(s-v)G\bigl(\delta_{v}\bigr)x\, dv\, ds.
\end{align}
Using \eqref{noviw} and \eqref{stari-novi}, we obtain that:
\begin{align*}
{\mathbf G}\bigl( \varphi^{\prime}\bigr)P_{r}x&=\lim_{t\rightarrow \infty}e^{-irt}\frac{1}{t}\Biggl[{\mathbf G}\bigl( \varphi^{\prime}\bigr)\int^{t}_{0}G\bigl(\delta_{s}\bigr)x\, ds\Biggr]
\\&+{\mathbf G}(I(\varphi))\lim_{t\rightarrow \infty}ire^{-irt}\frac{1}{t}\Bigl[G\bigl(\delta_{t}\bigr)x-x\Bigr]
\\&-r^{2}\lim_{t\rightarrow \infty}\frac{1}{t}\int^{t}_{0}e^{-irs}{\mathbf G}(I(\varphi))\Bigl[G\bigl(\delta_{s}\bigr)x-x\Bigr]\, ds.
\end{align*}
This equality in combination with \eqref{bvx} and the boundedness of function $G\bigl(\delta_{\cdot}\bigr)x$ shows that ${\mathbf G}( \varphi^{\prime})P_{r}x=-r^{2}{\mathbf G}(I(\varphi))P_{r}x.$ Since $\varphi \in {\mathcal D}_{0}$ was arbitrary, we get that
$AP_{r}x=-r^{2}P_{r}x,$ finishing the proof.
\end{proof}

Concerning the subspace weak almost periodicity of $C$-distribution semigroups, we have the following proposition.

\begin{proposition}\label{goldbergi}
Suppose that ${\mathcal G}$ is a \emph{(C-DS)} generated by $A,$ and ${\mathcal G}$ is $\tilde{E}$-weakly almost periodic for some linear subspace $\tilde{E}$ of $Z(A)$. If $\tilde{E}$ is dense in $E$ and $E$ is weakly sequentially complete, then the set $D$ defined above is total in $E$ and  ${\mathcal G}$ is $span(D)$-almost periodic.
\end{proposition}

\begin{proof}
Applying the partial integration, we easily get that
\begin{align*}
(ir-A)\frac{1}{t}\int^{t}_{0}e^{-irs}G\bigl( \delta_{s} \bigr)x\, ds=\frac{1}{t}\bigl[ x-e^{-irt} G\bigl( \delta_{t} \bigr)x\bigr],\quad t>0,\ r\in {\mathbb R},\
x\in \tilde{E}.
\end{align*}
Then we can proceed as in the proofs of \cite[Theorem 3]{bart} and \cite[Theorem 2.1]{zhengliu} in order to see that the set $D$ is total in $E.$ Hence, the statement of proposition is a simple consequence of Proposition \ref{spremte-se}(i).
\end{proof}

Concerning the subspace weak almost periodicity of $C$-distribution cosine functions, we would like to propose the following problem (cf. also \cite[Theorem  1.5, pp. 246-247]{x263}):
Suppose that ${\mathbf G}$ is a $(C-DCF)$ generated by $A,$ and ${\mathbf G}$ is $\tilde{E}$-almost periodic for some linear subspace $\tilde{E}$ of $Z_{2}(A)$. If $\tilde{E}$ is dense in $E$ and $E$ is weakly sequentially complete, is it true that the set $H$ is total in $E$ and that ${\mathbf G}$ is $span(H)$-almost periodic?

In \cite{aot}, we have also investigated almost periodic (degenerate) strongly continuous semigroups, cosine operator functions and the associated sine operator functions acting on Banach spaces which do not contain an isomorphic copy of the space
$c_{0}$ (cf. Theorem \ref{svinja}(vi)). The interested reader may try to formulate an analogue of Theorem \ref{prcko}(ii) in the case that ${\mathbf G}$ is a $(C-DCF)$ generated by $A$ and the mapping $t\mapsto \int^{t}_{0}G(\delta_{s})x,$ $t\geq 0$ is $\tilde{E}$-almost periodic for all elements $x$ belonging to some linear subspace $\tilde{E}$ of $Z_{2}(A)$.

We close the paper with the observation that the assertions of Proposition \ref{spremte-se}, Proposition \ref{polugrupe}, Theorem \ref{prcko} and Proposition \ref{goldbergi} hold for $C$-ultradistribution semigroups of $\ast$-class and $C$-ultradistribution cosine functions of $\ast$-class
in Banach spaces (cf. \cite{FKP} for the notion).

\end{document}